\documentclass{amsart} 

\usepackage{amssymb}
\usepackage{color}
\usepackage{cleveref}

\newcommand{\seq}[1]{{\left\langle{#1}\right\rangle}}
\newcommand{\uhr}[1]{\! \upharpoonright_{#1}}

\newcommand{\w}{\omega}

\newcommand{\cat}{\widehat{\phantom{\alpha}}}
\renewcommand \phi {\varphi}

\newcommand{\QQ}{\mathbb{Q}}

\newcommand{\notion}{low for paths}

\renewcommand{\P}{\mathcal{P}}
\newcommand{\Q}{\mathcal{Q}}
\newcommand{\R}{\mathcal{R}}

\newcommand{\M}{\mathcal{M}}
\newcommand{\N}{\mathcal{N}}

\newcommand{\iso}[2]{\text{Iso}(#1, #2)}

\theoremstyle{plain}
\newtheorem{theorem}{Theorem}

\newtheorem{lemma}[theorem]{Lemma}

\newcounter{claimCounter}[theorem]

\newtheorem{claim}[claimCounter]{Claim}

\theoremstyle{definition}
\newtheorem{definition}[theorem]{Definition}

\begin{document}

\title{Taking the path computably travelled}
\author[Franklin]{Johanna N.Y.\ Franklin}
\address{Department of Mathematics \\ Room 306, Roosevelt Hall \\ Hofstra University \\ Hempstead, NY 11549-0114 \\ USA}
\email{johanna.n.franklin@hofstra.edu}
\thanks{The first author was supported in part by Simons Foundation Collaboration Grant \#420806.}

\author[Turetsky]{Dan Turetsky}
\address{Department of Mathematics \\ Victoria University of Wellington \\ Wellington, New Zealand}
\email{dan.turetsky@vuw.ac.nz}

\date{\today}

\begin{abstract}
We define a real $A$ to be low for paths in Baire space (or Cantor space) if every $\Pi^0_1$ class with an $A$-computable element has a computable element. We prove that lowness for paths in Baire space and lowness for paths in Cantor space are equivalent and, furthermore, that these notions are also equivalent to lowness for isomorphism.
\end{abstract}

\maketitle

\section{Introduction}

Lowness notions are common objects of study in computability theory.  Examples include lowness and superlowness in degree theory, lowness for randomness, lowness for genericity, array computability, jump-traceability, and lowness for isomorphism and lowness for categoricity in computable structure theory.  Each of these notions characterizes a class of reals which are in some way no more useful as an oracle than the empty set.

One way to understand such notions is via tasks and instances: a real $A$ satisfies the lowness notion associated with a task if every instance of the task which has an $A$-computable solution also has a computable solution.  For example, we can consider the original lowness notion: a real $A$ is {\em low} if $A' \le_T \emptyset'$, but by the Schoenfeld Limit Lemma, this can be understood in our framework by saying that an instance is an $X \in 2^\omega$ and the task is to compute a limit approximation to $X$.  Thus, $A$ is low if and only if every $X \in 2^\omega$ which is limit-computable from $A$ is limit-computable from $\emptyset$.

Another well-known example is lowness for Martin-L\"{o}f randomness (see the text by Downey and Hirschfeldt~\cite{dhbook} for background on algorithmic randomness).  Here, an instance is again an $X \in 2^\omega$, and the task is to derandomize $X$, i.e., to capture $X$ with a Martin-L\"of test.  A real $A$ is {\em low for randomness} if every $X \in 2^\omega$ which $A$ can derandomize is already derandomized by $\emptyset$.

In this vein, Franklin and Solomon initiated the study of lowness for isomorphism, where an instance is a pair of computable structures and the task is to compute an isomorphism between the structures \cite{fs-lowim}.
\begin{definition}
A real $A$ is {\em low for isomorphism} if every pair of computable structures with an $A$-computable isomorphism between them have a computable isomorphism between them.
\end{definition}
We refer the reader to the text by Ash and Knight~\cite{ashknight} for background on computable structures. Franklin and Solomon showed that nontrivial examples of such reals exist.  For example, they showed that if $A$ is 2-generic, then $A$ is low for isomorphism \cite{fs-lowim}; an extension of this result was given by Franklin and Turetsky in \cite{ft-1genlow}.

Note that $A$ being low as described by some task does not mean that every $A$-computable solution to an instance of the task is itself a computable solution.  An instance may have multiple solutions, some of which are $A$-computable but not computable.  However, in such a case, the instance will also have computable solutions.  For example, consider two computable copies of $(\QQ, <)$.  It is a simple exercise to show that these two copies have an isomorphism in every Turing degree.  Thus, even if $A$ is low for isomorphism, there will be an $A$-computable isomorphism between these copies which is not a computable isomorphism.  However, there is also a computable isomorphism between these two copies.

We observe that we can understand the collection of isomorphisms between two computable structures as a $\Pi^0_1$-class in Baire space.
\begin{definition}
For computable structures $\M$ and $\N$, define
\[
\iso\M\N = \{ (f_0,f_1) \in \omega^\omega \times \omega^\omega : \M \cong_{f_0} \N \ \& \ f_1 = f_0^{-1}\}.
\]
\end{definition}
In general, the statement that a function is surjective is $\Pi^0_2$, but we avoid that difficulty by including the function's inverse: we can write down a $\Pi^0_1$-formula for $\iso\M\N$ by stating that $f_0$ is an embedding of $\M$ into $\N$ and that $f_1 = f_0^{-1}$.  The existence of the inverse ensures that $f_0$ is surjective.

As any isomorphism from $\M$ to $\N$ computes its inverse, we can understand lowness for isomorphism as the lowness notion in which the instances are $\Pi^0_1$-classes of the form $\iso\M\N$ and the task is to compute an element of the class.  It is then natural to consider the related lowness notion in which an instance is any $\Pi^0_1$-class and the task is to compute an element of the class.  This gives \emph{a priori} two notions, depending on whether one considers $\Pi^0_1$-classes in Baire space or in Cantor space.

\begin{definition}
A real $A$ is {\em \notion \ for Baire space} (or {\em \notion \ for Cantor space}) if every $\Pi^0_1$ class $\P \subseteq \omega^\omega$ (respectively, $\P \subseteq 2^\omega$) with an $A$-computable element has a computable element.
\end{definition}

We prove the following:
\begin{theorem}\label{thm:main}
For $A \in 2^\omega$, the following are equivalent:
\begin{enumerate}
\item $A$ is \notion \ for Baire space;
\item $A$ is \notion \ for Cantor space;
\item $A$ is low for isomorphism.
\end{enumerate}

\end{theorem}

$(1) \Rightarrow (2)$ is obvious, as every $\Pi^0_1$-class in Cantor space is itself a $\Pi^0_1$-class in Baire space.  $(1) \Rightarrow (3)$ follows from our discussion of $\iso\M\N$.  We will show $(2) \Rightarrow (1)$ and then $(3) \Rightarrow (2)$. 

The proof of $(2) \Rightarrow (1)$ relies on the following result of Simpson~\cite{Simpson04}:
\begin{lemma}\label{lem:Simpson}
If $\P \subseteq \omega^\omega$ and $\Q \subseteq 2^\omega$ are nonempty $\Pi^0_1$-classes, then there is a $\Pi^0_1$-class $\R \subseteq 2^\omega$ with $\R \equiv_w \P \cup \Q$.
\end{lemma}
Here, $\equiv_w$ is Muchnik equivalence: every element of $\R$ computes an element of $\P \cup \Q$, and every element of $\P \cup \Q$ computes an element of $\R$.

\begin{proof}[Proof of \Cref{thm:main}, $(2) \Rightarrow (1)$]
Suppose $A$ is \notion \ for Cantor space, and let $\P \subseteq \omega^\omega$ be a $\Pi^0_1$-class with an $A$-computable element $f$.  We must show that $\P$ has a computable element.  Fix $Q \subseteq 2^\omega$ a nonempty $\Pi^0_1$-class with no computable elements, e.g., the completions of Peano arithmetic, and let $\R$ be as in \Cref{lem:Simpson}.

As $f \in \P \subseteq \P \cup \Q$ and $\R \equiv_w \P \cup \Q$, there is $X \in \R$ with $X \le_T f \le_T A$.  As $A$ is \notion \ for Cantor space, there must be a computable $Y \in \R$.  Using again $\R \equiv_w \P \cup \Q$, there must be $g \in \P \cup \Q$ with $g \le_T Y$, and so $g$ is computable.  Since $\Q$ contains no computable elements, $g \in \P$ is a computable element as desired.
\end{proof}

As being \notion \ for Baire space is equivalent to being \notion \ for Cantor space, we shall refer to them both as simply {\em \notion}.

Now we turn to the proof of $(3) \Rightarrow (2)$. The following lemma is the heart of this result.

\begin{lemma}\label{lem:primary_construction}
For every $\Pi^0_1$ class $\Q \subseteq 2^\w$, there are computable structures $\M$ and $\N$ such that $\Q \equiv_w \iso\M\N$.
\end{lemma}

\begin{proof}
Fix a computable tree $T \subseteq 2^{<\w}$ such that $[T] = \Q$ and so that if $\sigma \in T$ is not a leaf, then both $\sigma\cat 0$ and $\sigma\cat 1$ are in $T$.  We can ensure $T$ has this property by replacing it with $\{\sigma, \sigma\cat 0, \sigma\cat 1 : \sigma \in T\}$.  This is a computable tree with this property, and it does not change $[T]$.

The language of our structures will consist of unary relation symbols $R_\sigma$ for $\sigma \in T$, a unary relation symbol $L$, a ternary relation symbol $P$, and a constant symbol $c$.  The universe for both $\M$ and $\N$ will be $T \times 2 = \{ (\sigma, i) : \sigma \in T, i \in \{0, 1\}\}$ (we caution the reader not to confuse $(\sigma, 0)$ with $\sigma\cat 0$). $\M$ and $\N$ will be identical in all ways except for their interpretations of $c$.

In both $\M$ and $\N$, 
\begin{itemize}
\item $R_\sigma( (\tau, i) )$ will hold if and only if $\sigma = \tau$,
\item $L( (\tau, i) )$ will hold if and only if $\tau$ is a leaf of $T$ and $i = 0$, and 
\item $P( (\tau_0, i_0), (\tau_1, i_1), (\tau_2, i_2) )$ will hold if and only if the following conditions are satisfied:
\begin{itemize}
\item $\tau_1 = \tau_0\cat 0$;
\item $\tau_2 = \tau_0\cat 1$; and
\item $i_0 + i_1 + i_2$ is even.
\end{itemize}
\end{itemize}

Finally, in $\M$, $c$ refers to the element $( \seq{}, 0)$, but in $\N$, $c$ refers to the element $( \seq{}, 1)$.

\medskip

Now, suppose $f: \M \to \N$ is an isomorphism.  Because of $R_\sigma$, it must be that 
\[
f\{ (\sigma, 0), (\sigma,1)\} = \{ (\sigma, 0), (\sigma, 1)\}
\]
for every $\sigma$.  We will say that {\em $f$ swaps $\sigma$} if $f( (\sigma, 0) ) = (\sigma, 1)$ (and, thus, if $f( (\sigma, 1) ) = (\sigma, 0)$).

\begin{claim}\label{claim:swapping_count}
Suppose $f: \M \to \N$ has the property that for every $\sigma \in T$, 
\[
f\{ (\sigma, 0), (\sigma,1)\} = \{ (\sigma, 0), (\sigma, 1)\}.
\]
Then $f$ respects $P$ if and only if, for every $\sigma$ not a leaf of $T$, $f$ swaps either 0 or 2 of $\{\sigma, \sigma\cat 0, \sigma\cat 1\}$.
\end{claim}

\begin{proof}
Suppose $f((\tau, 0)) = (\tau, j_\tau)$ for $\tau \in T$.  Then $f$ respects $P$ precisely if, for every $\sigma \in T$ not a leaf, $j_\sigma + j_{\sigma\cat0} + j_{\sigma\cat 1}$ is even.  This will be even if and only if 0 or 2 of the $j$ are 1, which holds if and only if $f$ swaps 0 or 2 of $\{\sigma, \sigma\cat 0, \sigma\cat 1\}$.
\end{proof}

We may make the following observations about an isomorphism $f: \M \to \N$:
\begin{itemize}
\item Because of $c^\M$ and $c^\N$, $f$ must swap $\seq{}$.
\item Because of $L$, $f$ must not swap any leaf $\sigma \in T$.
\item By \Cref{claim:swapping_count}, if $f$ swaps $\sigma$, it must swap exactly one of $\sigma\cat 0$ or $\sigma\cat 1$.
\end{itemize}

So from any isomorphism $f: \M \to \N$, we can recursively compute a $g \in \Q$ by ``following the swaps.''  More precisely:
\begin{enumerate}
\item We define $g\uhr{0} = \seq{}$.
\item If $\sigma = g\uhr{n}$, then inductively we know that $f$ swaps $\sigma$, and thus $\sigma$ is not a leaf of $T$.  Let $i$ be the unique element of $\{0, 1\}$ such that $f$ swaps $\sigma\cat i$.  Define $g\uhr{n+1} = \sigma\cat i$.
\end{enumerate}
This shows that $\Q \le_w \iso\M\N$.

Conversely, suppose $g \in \Q$.  We wish to compute an isomorphism $f: \M \to \N$ from $g$.  We will define $f$ by swapping along $g$ and fixing its values elsewhere.  More precisely:
\begin{enumerate}
\item If $\sigma \prec g$, define $f( (\sigma, i) ) = (\sigma, 1-i)$ for $i < 2$.
\item If $\sigma \not \prec g$, define $f( (\sigma, i) ) = (\sigma, i)$ for $i < 2$.
\end{enumerate}
Clearly, $f$ respects each of the $R_\sigma$.  Since $\seq{} \prec g$, we have that $f(c^\M) = c^\N$.  For any leaf $\sigma \in T$, we know that $\sigma \not \prec g$, so $f( (\sigma, 0) ) = (\sigma, 0)$ and $f$ respects $L$.  For any $\sigma \not \prec g$, $f$ does not swap any of $\{\sigma, \sigma\cat 0, \sigma\cat 1\}$; for any $\sigma \prec g$, $f$ swaps $\sigma$ and precisely one of $\sigma\cat 0$, $\sigma\cat 1$.  Thus, by \Cref{claim:swapping_count}, $f$ respects $P$.  Therefore, $f$ is an isomorphism, and this shows that $\iso\M\N \le_w \Q$.
\end{proof}

Now we can complete the proof of our result.

\begin{proof}[Proof of \Cref{thm:main}, $(3) \Rightarrow (2)$]
Suppose $A$ is low for isomorphism, and let $\P \subseteq 2^\omega$ be a $\Pi^0_1$-class with an $A$-computable element $X$.  We must show that $\P$ has a computable element.  By \Cref{lem:primary_construction}, we can fix computable structures $\M$ and $\N$ with $\P \equiv_w \iso\M\N$.  Then there is $f \in \iso\M\N$ with $f \le_T X \le_T A$.  As $A$ is low for isomorphism, there is a computable $g \in \iso\M\N$.  Then there is $Y \in \P$ with $Y \le_T g$, and $Y$ is the desired computable element.
\end{proof}

This result has a very pleasing corollary. Franklin and McNicholl introduced the notion of lowness for isometry, where an instance is a pair of computable metric spaces and the task is to compute an isometry between the structures \cite{fm-lowisom}.
\begin{definition}
A real $A$ is {\em low for isometry} if every pair of computable structures with an $A$-computable isometry between them have a computable isometry between them.
\end{definition}
We refer the reader to Pour-El and Richards~\cite{prbook} for background on computable metric spaces.  McNicholl and Stull have further studied this in the particular case where the metric spaces are Banach spaces~\cite{ms-isomlp}. Franklin and McNicholl showed that a real is low for isomorphism if and only if it is low for isometry \cite{fm-lowisom}.

We observe that, given two computable metric spaces, one can construct the $\Pi^0_1$-class of all isometries between these two spaces in a way similar to that for the class $\iso\M\N$ for computable structures.  Thus we derive as a corollary one direction of Franklin and McNicholl's result: if $A$ is low for isomorphism, then it is low for isometry.

\end{document}